\documentclass{siamart251216}
\usepackage{silence}
\usepackage{amssymb}
\usepackage{caption}
\usepackage{subcaption}
\usepackage{tikz}
\usetikzlibrary{quotes}
\def\R{\mathbb{R}}
\renewcommand{\hat}{\widehat}
\renewcommand{\bar}{\overline}
\renewcommand{\phi}{\varphi}
\renewcommand{\epsilon}{\varepsilon}
\newsiamthm{thm}{Theorem}
\newsiamthm{prop}{Proposition}
\newsiamthm{lem}{Lemma}
\newsiamthm{cor}{Corollary}
\newsiamthm{conjecture}{Conjecture}
\newsiamthm{defn}{Definition}
\newsiamremark{rem}{Remark}
\headers{Analysis of a model of bacterial persister cells}{C. Li, T. Meadows, and T. Day}
\title{Analysis of persistence thresholds for a nonlocal PDE--ODE model of bacterial persister cells}
\author{
  Chongming Li\thanks{Queen's University, Department of Mathematics and Statistics
    (\email{li.chongming@queensu.ca}).}
  \and
  Tyler Meadows\thanks{Queen's University, Department of Mathematics and Statistics
    (\email{tyler.meadows@queensu.ca}).}
  \and
  Troy Day\thanks{Queen's University, Department of Mathematics and Statistics
    (\email{day@queensu.ca}).
    Also affiliated with Queen's University, Department of Biology.}
}
\begin{document}
\maketitle
\begin{abstract}
  Within many bacterial colonies, persister cells exist as a subpopulation that is tolerant to antibiotics and other stressors, yet not genetically distinct from the rest of the colony. A recent study has proposed epigenetic inheritance as a mechanism that leads to the presence of persister cells. We analyze a nonlocal PDE--ODE model introduced in that study to describe the epigenetic inheritance process and establish its mathematical well-posedness, including existence, uniqueness, and nonnegativity of solutions. We identify a sharp parameter threshold delineating extinction from persistence of the colony: below this threshold the washout equilibrium is globally asymptotically stable, while above it a unique positive equilibrium exists and the population is weakly persistent. Notably, this threshold is independent of the internal community structure.
\end{abstract}

\section{Introduction}
\label{sec:Introduction}

Bacterial persister cells are a subpopulation of some bacterial colonies, exhibiting resistance to certain stressors such as antibiotics and starvation. Persister cells were first discovered in 1944 by Joseph Bigger \cite{Bigger1944} during his experiments with penicillin. After treating a colony of bacteria with the drug, a small number of cells would survive. New colonies grown using the surviving cells were just as susceptible to antibiotics as the original colony.

Persister cells survive the application of antibiotics and other stressors by entering a dormant, or quiescent, state. Since most \(\beta\)-lactam antibiotics, like penicillin, act by disrupting cell wall formation during cell division, avoiding replication entirely is an effective strategy for surviving antibiotics. Similarly, shutting down most cellular functions serves as an effective survival strategy during times of starvation, oxidative stress, and other environmental stresses. 

Biological studies \cite{kim2007,moyed1983,kim2001,wolf1989} have shown that bacterial persistence is a phenotypic state rather than genetic adaptation; persister cells are genetically identical to the original wild-type bacteria, but differ in certain quantitative traits that enhance survival under stress. In mathematical models, these phenotypic traits are often represented by expression levels of particular chemicals, and it is assumed that bacteria switch to the persister state when the expression level exceeds a certain threshold~\cite{day2016interpreting}. 

Modelling persister cells via PDEs in the bacterial literature is relatively new as the existing models consider the interaction between persister and susceptible states of bacteria using ODEs (see \cite{balaban2004bacterial}). The advantage of PDE models is the feasibility of tracking the exchange between subpopulations indexed by a continuum of expression levels, together with the dormant persister cells, at any given time. Furthermore, Day in \cite{day2016interpreting} incorporated the birth-jump process via a nonlocal term. In spatial ecology, the term birth-jump process \cite{hill2015} is often used to describe situations in which newly produced offspring disperse from their original location immediately after birth. In our model, the change in phenotypes between generations can be interpreted, by analogy, as a birth-jump process.

We consider a model introduced in  \cite{day2016interpreting} consisting of an ODE coupled to a PDE with a nonlocal birth-jump term:
\begin{subequations}\label{eq:semipde1}
\begin{align}
    \partial_t n &=\chi_{\alpha}[bR-d]n-\mu R nb\chi_{\alpha}+\mu bR\int^{1}_{0} n(y,t)\chi_{\alpha}(y)p(x;y) dy-\partial_x[v(x)n]+m\partial_{xx}n,\label{eq:semipde1a}\\
    \partial_tR&=\theta-\eta R-bR\int^{1}_{0} \chi_{\alpha} n dx. \label{eq:semipde1b}
\end{align}
\end{subequations}
In system \eqref{eq:semipde1}, we denote by \(n:=n(x,t)\) and \(R:=R(t)\), respectively, the density of individuals with expression level \(x\in \Omega = (0,1)\) (normalized to the unit interval) and the amount of resource present at time \(t\in \R_{\ge 0}\).

The first term of \eqref{eq:semipde1a}, \(\chi_{\alpha}(x)[bR(t)-d]\), represents the net per capita reproduction rate. We suppose that the birth rate is proportional to available resource \(R(t)\), while the death rate is a fixed constant \(d\). The function 
\[\chi_\alpha:=\chi_\alpha(x)=\begin{cases} 1 & x<\alpha\\ 0 &x \ge \alpha\end{cases}\]
delineates between regular cells and persister cells; cells with expression levels above \(\alpha\) are dormant and hence do not replicate or die, while cells with expression levels below \(\alpha\) behave normally.

The second and third terms of \eqref{eq:semipde1a} are related to changes in expression levels across generations during the reproduction process. By the biological interpretation in \cite{day2016interpreting,sat2011,vee2008}, the expression level can be transferred from parents to offspring during reproduction. We denote by \(\mu\) the fixed probability of a change in expression level when a parent gives birth to an offspring. Inside the integral, \(p(x;y)\) is the probability density of the offspring's expression level being \(x\), given that the parent has expression level \(y\), and assuming that a change takes place during the process of reproduction.

The final two terms of \eqref{eq:semipde1a} 
are advection and diffusion terms describing within-generation change. The advection term is based on a mechanism of cellular homeostasis \cite{lib2023} while the diffusion represents within-generation noise. Cellular homeostasis arises from microbes regulating their internal environment through physiological processes. The expression level of an individual is typically centered around a value that supports replication. This behavior is captured by choosing \(v(x)\) to model within-generation adjustment. 

For the resource equation \eqref{eq:semipde1b}, the parameter \(\theta\) is a resource inflow rate whereas \(\eta\) is the per capita loss rate of the resource. The integral \(bR\int^{1}_{0} \chi_{\alpha}(x) n dx\) is the amount of resource consumed by the whole population of microbes per unit time. 

Given the above biological description, we make the following assumptions:
\begin{enumerate}
    \item [\(F_{1}:\)] \(v\in C^{1}(\overline{\Omega})\) with \(v(0)=v(1)=0\), and there exists a unique point \(\gamma\in(0,1)\) with \(v(\gamma)=0\), called the \emph{homeostatic point}. A concrete example is the cubic \(v(x)=v_{0}\,x(x-\gamma)(x-1)\) with \(v_{0}\in\mathbb{R}\).
    \item [\(F_{2}:\)] \(p\in L^{2}(\Omega\times\Omega)\) with \(p\ge 0\) a.e.\ and \(\int_0^1 p(x;y)\,dx=1\) for a.e.\ \(y\in\Omega\).
    \item [\(F_{3}:\)] \(b,d,\theta,\eta,m>0\) and \(\mu\in[0,1]\).
\end{enumerate}
Assumptions \(F_1\)--\(F_3\) are in force throughout the paper and will not be repeated in the statements of our results.

Lastly, we assume the following boundary and initial conditions:
\begin{align}
\label{eq:semiinicon}
    \partial_x n(0,t)=\partial_x n(1,t) = 0\quad\text{and}\quad n(x,0)=n_{0}(x)\in L^{2}_{+}(\Omega),\quad R(0)=R_{0}\in \R_{\ge 0}.
\end{align}

The main objective of this paper is to conduct a rigorous mathematical analysis of the nonlocal PDE–ODE system \eqref{eq:semipde1}–\eqref{eq:semiinicon}. The system shares the resource-consumer coupling of the classical chemostat \cite{SmithWaltman1995}, for which the dichotomy between extinction and persistence of the microbial population is well understood. Here we establish the analogous threshold for the structured model \eqref{eq:semipde1}, where the discontinuous coefficient \(\chi_\alpha\), the nonlocal birth-jump operator, and the advection term with degenerate boundary coefficients preclude classical solutions around \(x=\alpha\). For this reason, we formulate the PDE component of \eqref{eq:semipde1} in the framework of \emph{mild solutions} and \(C_0\)-semigroup theory. This approach allows us to treat the linear diffusion-advection-reaction operator as the generator of a positive \(C_0\)-semigroup and to incorporate the nonlinear and nonlocal terms via the variation of constants formula.

Our main results are as follows. We prove that system \eqref{eq:semipde1}--\eqref{eq:semiinicon} is globally well-posed for nonnegative initial data, with solutions remaining nonnegative for all time (Theorem~\ref{thm:global-solution}). The asymptotic behavior is governed by a sharp threshold: when \(\theta/\eta<d/b\), the washout equilibrium \((0,\theta/\eta)\) is the unique nonnegative steady state and is globally asymptotically stable (Theorem~\ref{thm:washout}); when \(\theta/\eta>d/b\), a unique positive equilibrium exists (Theorem~\ref{thm:positive-equilibrium}) and the population is weakly persistent (Theorem~\ref{thm:weak-persistence}). Notably, this threshold depends only on the resource and demographic parameters and is independent of the parameters governing the internal phenotypic dynamics (see Remark~\ref{rem:threshold-independence}).

The paper is organized as follows: in Section \ref{sec:WellPosedness}, we investigate the well-posedness of system \eqref{eq:semipde1} by reframing the equations in a way that permits the use of semigroup theory. In Section \ref{sec:AsymptoticBehavior} we analyze the long-term behavior of solutions using spectral theory and adjoint duality arguments. We show that for values of \(\theta/\eta\) large enough, the system is weakly persistent; the total microbial biomass does not decay to zero as time tends to infinity. Finally, in Section \ref{sec:Numerics}, we illustrate our findings with numerical simulations, which help support the conjecture that the positive steady state is locally asymptotically stable when the system is weakly persistent.

\section{Well-Posedness of the Model} \label{sec:WellPosedness}

Since \(\chi_\alpha(x)\) is discontinuous, we work in an \(L^2(\Omega)\) framework and formulate solutions in the mild sense using semigroup theory. For \(T>0\), define the Banach spaces 
\[
\mathcal{H}_T:=C([0,T],L^2(\Omega)),\quad \text{and}\quad \mathcal{P}_T:=C[0, T],  
\]
with norms
\[
    \|n\|_{\mathcal H_T} := \sup_{t\in[0,T]}\|n(t)\|, \qquad
    \|R\|_{\mathcal P_T} := \sup_{t\in[0,T]}|R(t)|, 
\]
respectively. Throughout, \(\|\cdot\|\) and \(\langle\cdot,\cdot\rangle\) denote respectively the norm and inner product on \(L^2(\Omega)\); other norms carry explicit subscripts. For fixed \(R(t)\in \mathcal{P}_T\), equation \eqref{eq:semipde1a} can be seen as a non-autonomous, nonlocal parabolic PDE. The right hand side of \eqref{eq:semipde1a} may be split into an autonomous part,
    \begin{equation}
    \label{defA}
        \mathcal{A}n:=-d\,\chi_\alpha n-\partial_{x}[v(x)n]+m\partial_{xx}n, 
    \end{equation}
with domain 
\[
\mathcal{D}(\mathcal{A}) = \{ n \in H^2(\Omega) : \partial_x n(0) = \partial_x n(1) = 0 \}, 
\]
and a non-autonomous part \(bR(t)\mathcal{N}n\), where 
    \begin{equation}
    \label{defNon}
        \mathcal{N}n:=(1-\mu)\chi_\alpha n+\mu\int^{1}_{0} \chi_\alpha(y)n(y)p(x;y)\,dy, 
    \end{equation}
with domain \(\mathcal{D}(\mathcal{N}) = L^2(\Omega)\).

We begin with the semigroup properties of the operator \(\mathcal{A}\). 
\begin{lem}
\label{lem:PropertiesA}
The operator \(\mathcal{A}\) given in \eqref{defA} generates a quasi-contraction semigroup, that is, there is a constant \(\omega\in \mathbb{R}\) such that
\begin{equation}\label{eq:semigroup-estimate}
    \|e^{\mathcal A t }n\|
    \le e^{\omega t}\|n\|
\end{equation}
for all \(t>0\) and all \(n\in L^2(\Omega).\)
\end{lem}
\begin{proof}
To show that \(\mathcal{A}\) generates a quasi-contraction semigroup, we verify the three conditions of \cite[Theorem 9.5.2]{hillen2023}. The domain \(\mathcal{D}(\mathcal A)\) is dense in \(L^2(\Omega)\) because \(C_c^\infty(\Omega)\subset\mathcal{D}(\mathcal A)\) and \(C_c^\infty(\Omega)\) is dense in \(L^2(\Omega)\).

For quasi-dissipativity, we show that \(\operatorname{Re}\langle n, \mathcal{A}n\rangle \le \omega \|n\|^2\) for all \(n \in \mathcal{D}(\mathcal{A})\). By assumption \(F_1\) and integration by parts, the bilinear form associated with \(-\mathcal{A}\) on \(V=H^1(\Omega)\) is
\begin{align}
\label{eq:bilinear-form-a}
a(n,w)=\int_{0}^{1}\bigl[m\,\partial_{x}n\,\partial_{x}w+v(x)\,\partial_{x}n\;w+\bigl(v'(x)+\chi_{\alpha}(x)d\bigr)\,n\,w\bigr]\,dx.
\end{align}
Setting \(w=n\),
\begin{align*}
-\langle n, \mathcal{A}n \rangle &= -\int^{1}_{0} n \left( -\chi_{\alpha} d n - (vn)_x + m n_{xx} \right) dx \\
&= \int^{1}_{0}\left[\chi_{\alpha} d n^{2} - n_{x}v n + m |n_{x}|^{2}\right] dx.
\end{align*}
Applying Young's inequality to the advection term \(\int^{1}_{0} n_x v n\),
\begin{align*}
-\langle n, \mathcal{A}n \rangle &\ge d \int \chi_\alpha n^2dx - \|v\|_{L^\infty} \int |n| |n_x|dx + m \|n_x\|^2 \\
&\ge -\|v\|_{L^\infty} \left( \frac{\|n\|^2}{2\varepsilon} + \frac{\varepsilon \|n_x\|^2}{2} \right) + m \|n_x\|^2 \\
&= -\frac{\|v\|_{L^\infty}}{2\varepsilon} \|n\|^2 + \left(m-\frac{\varepsilon\|v\|_{L^\infty}}{2}\right) \|n_{x}\|^{2}.
\end{align*}
Letting \(\varepsilon = \frac{m}{\|v\|_{L^\infty}}\), the coefficient of \(\|n_x\|^2\) becomes \(m/2\), and we obtain
\begin{align}
\label{eq:<n,An>-inequality}
-\langle n, \mathcal{A}n \rangle \ge -\frac{\|v\|_{L^\infty}^2}{2m} \|n\|^2 + \frac{m}{2} \|n_x\|^2.
\end{align}
Rearranging gives \(\langle n, \mathcal{A}n \rangle \le \frac{\|v\|_{L^\infty}^2}{2m} \|n\|^2 - \frac{m}{2}\|n_x\|^2\). Defining \(\omega = \frac{\|v\|_{L^\infty}^2}{2m}\),
\[
\operatorname{Re}\langle n, \mathcal{A}n\rangle \le \omega\|n\|^2.
\]

It remains to verify the range condition. For \(\lambda > \omega\) and given \(f \in L^2(\Omega)\), consider the equation
\begin{align}
\label{eq:resolvent}
\lambda n - \mathcal{A}n = f.
\end{align}
The associated bilinear form \(B\, \colon H^1 \times H^1 \to \mathbb{R}\) is
\[
B[n,\phi] = \int_0^1 (\lambda n \phi - \phi \mathcal{A}n) dx = \int_0^1 \left[ (\lambda + d\chi_\alpha) n \phi + (vn)_x \phi + m n_x \phi_x \right] dx.
\]
Using the preceding estimate \eqref{eq:<n,An>-inequality}, for any \(n \in H^1(\Omega)\),
\[
B[n,n] = \lambda \|n\|^2 - \langle n, \mathcal{A}n \rangle \ge \left(\lambda - \frac{\|v\|_{L^\infty}^2}{2m}\right) \|n\|^2 + \frac{m}{2} \|n_x\|^2.
\]
For \(\lambda > \omega = \frac{\|v\|_{L^\infty}^2}{2m}\), there exists \(C>0\) such that \(B[n,n] \ge C \|n\|_{H^1}^2\), and \(B\) is coercive and bounded. Hence equation \eqref{eq:resolvent} has a unique weak solution \(n \in H^1(\Omega)\) by the Lax-Milgram Theorem \cite[Theorem~6.42]{san2022}. It follows that \(n \in H^2(\Omega)\) with the Neumann boundary conditions by elliptic regularity \cite[Theorem~8.30]{san2022}, hence \(n \in \mathcal{D}(\mathcal{A})\) and the range of \((\lambda - \mathcal{A})\) is all of \(L^2(\Omega)\).

All hypotheses of \cite[Theorem~9.5.2]{hillen2023} are satisfied, and therefore \(\mathcal{A}\) generates a quasi-contraction semigroup on \(L^2(\Omega)\).
\end{proof}

The semigroup \((e^{\mathcal{A}t})_{t\ge 0}\) has two further properties that will be used throughout.

\begin{lem}[Positivity and irreducibility of \(e^{\mathcal{A}t}\)]
\label{lem:Semigroup-A-Positive-Irreducible}
The semigroup \((e^{\mathcal{A}t})_{t\ge 0}\) generated by \(\mathcal{A}\) on \(L^2(\Omega)\) is positive and irreducible.
\end{lem}

\begin{proof}
All coefficients appearing in the bilinear form \eqref{eq:bilinear-form-a}---namely \(m>0\), \(v\in C^{1}(\bar{\Omega})\), \(v'\in C(\bar{\Omega})\), \(d>0\), and \(\chi_{\alpha}\in L^{\infty}(\Omega)\)---are real-valued. Since \(V=H^{1}(\Omega)\) corresponds to Neumann boundary conditions, the semigroup \((e^{\mathcal{A}t})_{t\ge 0}\) is positive by \cite[Corollary~4.3]{ouhabaz2005}.

For irreducibility, \(\Omega\) is connected, the coefficients are real-valued, and the lattice property \(n\in V\implies (\operatorname{Re}n)^{+}\in V\) for \(V=H^{1}(\Omega)\) holds by \cite[Proposition~4.4]{ouhabaz2005}. The irreducibility of \((e^{\mathcal{A}t})_{t\ge 0}\) then follows from \cite[Theorem~4.5]{ouhabaz2005}.
\end{proof}

The nonlocal operator \(\mathcal{N}\) defined in \eqref{defNon} has the following properties.
\begin{lem}[Boundedness and positivity of \(\mathcal N\)]
\label{lem:N-bounded}
The operator \(\mathcal N:L^2(\Omega)\to L^2(\Omega)\) is a bounded linear operator. Moreover, it is positive in the sense that
\(n\ge 0\) implies \(\mathcal N (n)\ge 0\) a.e..
\end{lem}

\begin{proof}
Linearity is immediate from equation \eqref{defNon}. To show \(\mathcal{N}\) is bounded, we estimate \(\mathcal{N}\) in two parts. Since \(\chi_\alpha\in L^\infty(\Omega)\) and
\(0\le 1-\mu\le 1\),
\begin{equation}\label{eq:chibound}
\|(1-\mu)\chi_\alpha n\|
\le (1-\mu)\|\chi_\alpha\|_{L^{\infty}}\|n\|
\le \|n\|.
\end{equation}
For the integral part, define the operator \(\mathcal{K}: L^2(\Omega)\to L^2(\Omega)\) by
\begin{equation}\label{eq:K_operator}
\mathcal{K}n(x)=\int_0^1 \chi_\alpha(y)n(y)p(x;y)\,dy.
\end{equation}

Using Fubini's theorem and the Cauchy--Schwarz inequality,
\begin{align*}
\|\mathcal Kn\|^{2}
&= \int_0^1 \Bigl|\int_0^1 \chi_\alpha(y)n(y)p(x;y)\,dy\Bigr|^{2} dx\\
&\le \int_0^1 \left(\int_0^1 |\chi_\alpha(y)n(y)|^{2}\,dy\right)
            \left(\int_0^1 |p(x;y)|^{2}\,dy\right) dx \\
&= \int_0^1 |\chi_\alpha(y)n(y)|^{2}\,dy
    \int_0^1 \left(\int_0^1 |p(x;y)|^{2}\,dy\right) dx.
\end{align*}
Setting \(C_p:=\|p\|_{L^{2}(\Omega\times\Omega)}^{2}\), we obtain
\begin{equation}\label{eq:Kbound}
\|\mathcal Kn\|^{2}
\le C_p \|\chi_\alpha n\|^{2}
\le C_p \|n\|^{2}.
\end{equation}
Hence \(\|\mathcal K\|\le \sqrt{C_p}\). Combining \eqref{eq:chibound} and \eqref{eq:Kbound},
\[
\|\mathcal N n\|
\le (1-\mu)\|\chi_\alpha n\| + \mu\|\mathcal Kn\|
\le \bigl((1-\mu)+\mu\sqrt{C_p}\bigr)\|n\|.
\]
 Positivity follows from \(n\ge 0\) with \(\chi_\alpha\ge 0\) and \(p\ge 0\), which gives
\[
(\mathcal N n)(x)
= (1-\mu)\chi_\alpha(x)n(x)
+ \mu\int_0^1 \chi_\alpha(y)n(y)p(x;y)\,dy \ge 0
\]
for a.e. \(x\in \Omega\). 
\end{proof}

We now rewrite system \eqref{eq:semipde1}--\eqref{eq:semiinicon} in its mild form.

\begin{defn}[Mild solution]\label{def:semi-mild-solution}
Let \(0<\tau<T\). We say that a pair \((n,R)\in \mathcal H_\tau\times\mathcal P_\tau\) is a
\emph{mild solution} of \eqref{eq:semipde1} on \([0,\tau]\)
with initial data \((n_0,R_0)\) if for all \(t\in[0,\tau]\),
\begin{subequations}\label{eq:semi-mild-system}
\begin{align}
    n(t)
    &= e^{\mathcal A t} n_0
      + \int_0^t e^{\mathcal A (t-s)}\,b\,R(s)\,\mathcal N\bigl[n(s)\bigr]\;ds,
        \label{eq:semi-mild-n}\\[1ex]
    R(t)
    &= R_0 + \int_0^t
      \Bigl(
          \theta - \eta R(s)
          - b R(s)\int_0^1 \chi_\alpha(x)\,n(x,s)\,dx
      \Bigr)\,ds.
        \label{eq:semi-mild-R}
\end{align}
\end{subequations}
\end{defn}

\begin{thm}[Local existence and uniqueness of mild solutions]
\label{thm:local-wellposedness}
Assume that \(F_1\)--\(F_3\) hold and that \((n_0,R_0)\) satisfy \eqref{eq:semiinicon}. Then there exists a time \(\tau>0\) such that \eqref{eq:semipde1} has a unique mild solution on \([0,\tau]\) with initial data \((n_0,R_0).\)
\end{thm}

\begin{proof}
For any \(\tau>0\), the product space \(X_\tau := \mathcal H_\tau\times\mathcal P_\tau\) with norm \(\|(n,R)\|_{X_\tau} := \|n\|_{\mathcal H_\tau} + \|R\|_{\mathcal P_\tau}\) is a Banach space. We define an operator 
\(\Phi:X_\tau \to X_\tau\) component-wise by \(\Phi(n,R) = (\Phi_1(n,R),\Phi_2(n,R)),\) where
\begin{subequations}\label{eq:Phi-definition}
\begin{align}
    \Phi_1(n,R)(t)
    &:= e^{\mathcal A t}n_0
      + \int_0^t e^{\mathcal A (t-s)}\,b\,R(s)\,\mathcal N\bigl[n(s)\bigr]\;ds,
        \quad t\in[0,\tau],\label{eq:Phi1}\\[1ex]
    \Phi_2(n,R)(t)
    &:= R_0 + \int_0^t
      \Bigl(
          \theta - \eta R(s)
          - b R(s)\int_0^1 \chi_\alpha(x)\,n(x,s)\,dx
      \Bigr)\,ds,
        \quad t\in[0,\tau].\label{eq:Phi2}
\end{align}
\end{subequations}

Fixed points of \(\Phi\) in \(X_\tau\) are precisely mild solutions
of \eqref{eq:semipde1} on \([0,\tau]\) with initial data \((n_0,R_0).\) For \(r>0\), set \(M = \|n_0\|+|R_0| + r\) and consider the closed ball
\[
B_r(n_0,R_0)= \{(n,R)\in X_\tau : \| (n,R)-(n_0,R_0)\|_{X_\tau} \le r \}.
\]
For any \((n,R)\in B_r(n_0,R_0)\), one has \(\|n\|_{\mathcal{H}_\tau}\le M\) and \(\|R\|_{\mathcal{P}_\tau}\le M\), and the semigroup estimate \eqref{eq:semigroup-estimate} together with Lemma~\ref{lem:N-bounded} gives
\begin{align*}
\|\Phi_1(n, R)(t) - n_0\| &\le \|e^{\mathcal{A}t}n_0 - n_0\| + \int_0^t b |R(s)|\|e^{\mathcal{A}(t-s)}\mathcal{N}n(s)\| ds\\
 &\le \epsilon(\tau) + \tau e^{\omega\tau} b M C_\mathcal{N} M,
\end{align*}
where \(C_\mathcal{N} = (1-\mu)+\mu\sqrt{C_p}\) and \(\epsilon(\tau) := \sup_{t\in[0,\tau]}\|e^{\mathcal{A}t}n_0 - n_0\| \to 0\) as \(\tau \to 0\) by the strong continuity of \((e^{\mathcal{A}t})_{t\ge0}\). Similarly,
\begin{align*}
|\Phi_2(n, R)(t) - R_0| &\le \int_0^t \bigl(\theta + \eta|R(s)| + b|R(s)| \|\chi_{\alpha}\| \|n(s)\|\bigr) ds\\
&\le \tau (\theta + \eta M + b M^2)
\end{align*}
for all \(t\in[0,\tau]\). Summing these estimates and choosing \(\tau\) small enough gives \(\|\Phi(n, R) - (n_0, R_0)\|_{X_{\tau}} \le r\), hence \(\Phi (B_r(n_0,R_0)) \subseteq B_r(n_0,R_0)\).

We now verify that \(\Phi\) is a contraction on \(B_r(n_0,R_0)\).

Let \((n_1, R_1), (n_2, R_2) \in B_r(n_0,R_0)\). Then,
\begin{align}
  &\|\Phi_1(n_1, R_1) - \Phi_1(n_2, R_2)\|_{\mathcal{H}_\tau}\nonumber\\
  &= \sup_{t \in [0,\tau]}\biggl\| \int_0^t be^{\mathcal{A}(t-s)}\bigl(R_1(s)\mathcal{N}[n_1(s)]-R_2(s)\mathcal{N}[n_2(s)]\bigr)ds\biggr\| \nonumber\\
  &\le \sup_{t \in [0,\tau]} \int_0^t e^{\omega(t-s)} b \bigl( |R_1-R_2| C_\mathcal{N} \|n_1\| + |R_2| C_\mathcal{N} \|n_1-n_2\| \bigr) ds \nonumber\\
  &\le \tau e^{\omega\tau} bC_\mathcal{N} M \bigl( \|R_1-R_2\|_{\mathcal{P}_\tau} + \|n_1-n_2\|_{\mathcal{H}_\tau} \bigr). \label{eq:Phi1-contraction}
\end{align}
Similarly, 
\begin{align*}
  &\|\Phi_2(n_1, R_1) - \Phi_2(n_2, R_2)\|_{\mathcal{P}_\tau}\\
  &= \sup_{t\in[0,\tau]} \biggl| \int_0^t \biggl(\eta (R_1(s)-R_2(s)) + b\int_0^1\chi_\alpha(x)\bigl(R_1(s)n_1(x,s)-R_2(s)n_2(x,s)\bigr)dx\biggr)ds \biggr| \\
  &\le \tau \bigl( \eta \|R_1-R_2\|_{\mathcal{P}_\tau} + b ( \|R_1-R_2\|_{\mathcal{P}_\tau} \|n_1\|_{\mathcal{H}_\tau} + \|R_2\|_{\mathcal{P}_\tau} \|n_1-n_2\|_{\mathcal{H}_\tau} ) \bigr) \\
  &\le \tau \bigl( (\eta + bM) \|R_1-R_2\|_{\mathcal{P}_\tau} + bM \|n_1-n_2\|_{\mathcal{H}_\tau} \bigr).
\end{align*}
Summing the inequalities gives
\[
\|\Phi(n_1, R_1) - \Phi(n_2, R_2)\|_{X_{\tau}} \le \tau e^{\omega\tau} C_{M} \|(n_1, R_1) - (n_2, R_2)\|_{X_{\tau}},
\]
where \(C_{M}\) depends on the constants \(M, b, \eta, C_{\mathcal{N}}\).
By choosing \(\tau\) such that \(\tau e^{\omega\tau}C_{M} < 1\), the operator \(\Phi\) is a strict contraction on \(B_r(n_0,R_0)\). The Banach fixed-point theorem then gives a unique fixed point \((n, R) \in B_r(n_0,R_0)\), which is the unique mild solution of \eqref{eq:semipde1} on \([0, \tau]\) with initial data \((n_0,R_0)\).
\end{proof}

In order to extend the local solution to all \(t\ge 0\), the following nonnegativity result is required.

\begin{lem}[Positivity of solutions]\label{lem:positivity}
Let \(\tau>0\) be the local existence time produced by Theorem~\ref{thm:local-wellposedness}, and let \((n,R)\) be the unique mild solution on \([0,\tau]\) with initial data \(n_0\in L^2_+(\Omega)\) and \(R_0\ge0\) as in \eqref{eq:semiinicon}. Then
\[
n(x,t)\ge0 \quad\text{for a.e. }x\in\Omega,\ t\in[0,\tau],
\qquad
R(t)\ge0 \quad\text{for all }t\in[0,\tau].
\]
\end{lem}

\begin{proof}
Define the integrating factor
\[
    \Psi(t) := \exp\!\left(\eta t + b\int_0^t\int_0^1 \chi_\alpha(x)n(x,s)\,dx\,ds\right) > 0,
    \qquad t\in[0,\tau]
\]
and set \(Y(t) = \Psi(t)R(t)\). Using \eqref{eq:semipde1b}, a direct computation gives
\[
\begin{aligned}
    \dot{Y}(t)
    &= \dot{\Psi}(t)R(t) + \Psi(t)\dot{R}(t) \\
    &= \Psi(t)\bigg(\eta + b \!\int_0^1 \chi_\alpha n(x,t)\,dx\bigg)R(t)
       + \Psi(t)\bigg(\theta - \eta R(t) - b R(t)\!\int_0^1 \chi_\alpha n(x,t)\,dx\bigg) \\
    &= \Psi(t)\,\theta \;\ge\; 0.
\end{aligned}
\]
Hence \(Y\) is nondecreasing on \([0,\tau]\). Since \(Y(0)=R_0\ge0\) and \(\Psi(t)>0\), it follows that \(R(t) = Y(t)/\Psi(t) \ge 0\) for all \(t\in[0,\tau]\).

For the positivity of \(n\), let
\[
\mathcal{H}_\tau^+:=\{n\in\mathcal H_\tau : n(x,t)\ge 0 \text{ for a.e.\ }x\in\Omega,\ \forall t\in[0,\tau]\}.
\]
Define iterates \(n^{(0)}(t) = n_0\) and \(n^{(k+1)}(t) = \Phi_1(n^{(k)},R)\), where \(\Phi_1\) is given by \eqref{eq:Phi1}. Since \(n_0\in L^2_+(\Omega)\), we have \(n^{(0)}\in\mathcal{H}_\tau^+\). Moreover, since \((e^{\mathcal A t})_{t\ge0}\) is a positive semigroup by Lemma~\ref{lem:Semigroup-A-Positive-Irreducible}, \(\mathcal{N}\) is a positive operator by Lemma~\ref{lem:N-bounded}, and \(R(t)\ge 0\), if \(n^{(k)}\in\mathcal{H}_\tau^+\), then
\[
n^{(k+1)}(t) = e^{\mathcal{A}t}n_0+\int_0^t e^{\mathcal{A}(t-s)}bR(s)\mathcal{N}[n^{(k)}(s)]\,ds \ge 0,
\]
and hence \(n^{(k)}\in\mathcal{H}_\tau^+\) for all \(k\ge 0\) by induction. Setting \(R_1=R_2=R\) in \eqref{eq:Phi1-contraction}, the map \(\Phi_1(\cdot,R)\) is a strict contraction on \(\mathcal{H}_\tau\), and the iterates \(\{n^{(k)}\}\) converge in \(\mathcal{H}_\tau\) to the unique mild solution \(n\). Since \(\mathcal{H}_\tau^+\) is a closed subset of \(\mathcal{H}_\tau\), we conclude that \(n\in\mathcal{H}_\tau^+\).
\end{proof}

\begin{thm}[Global well-posedness of mild solutions]\label{thm:global-solution}
Let \((n_0,R_0)\) satisfy \eqref{eq:semiinicon} with
\(n_0\in L^2_+(\Omega)\) and \(R_0\ge0\).
Then system \eqref{eq:semipde1} has a unique mild solution \((n,R)\) on \([0,\infty)\) satisfying
\(
n(x,t)\ge0\) for a.e.\ \(x\in\Omega\) and
\(R(t)\ge0\) for all \(t\ge0.\)
\end{thm}

\begin{proof}
By Theorem~\ref{thm:local-wellposedness}, a unique mild solution \((n,R)\) exists on a maximal interval \([0,T_{\max})\) for some \(T_{\max}\in(0,\infty]\). Lemma~\ref{lem:positivity} gives \(n\ge 0\) a.e.\ and \(R\ge 0\) on the local interval \([0,\tau]\) given by Theorem~\ref{thm:local-wellposedness}. Since \((n(\tau),R(\tau))\in L^2_+(\Omega)\times\mathbb R_+\), the argument may be repeated on successive intervals, and by induction \(n\ge 0\) a.e.\ and \(R\ge 0\) on \([0,T_{\max})\).
We now prove that necessarily \(T_{\max}=\infty\).

Since \(n(t)\ge 0\) for all \(t\), we have from \eqref{eq:semipde1b} that \(\dot R(t) \le \theta - \eta R(t)\) on \([0, T_{\max})\). By the comparison principle for scalar ODEs,
\begin{align}
\label{eq:range-of-R(t)}
0 \le R(t) \le \max\Bigl\{R_0,\ \frac{\theta}{\eta}\Bigr\}
    \qquad\text{for all }t\in[0,T_{\max}).
\end{align}
Setting \(R_{\max}:=\max\{R_0,\theta/\eta\}\), the mild equation \eqref{eq:semi-mild-n} gives
\[
\|n(t)\| \le e^{\omega t}\|n_0\| + bR_{\max}C_\mathcal{N}\int_0^t e^{\omega(t-s)}\|n(s)\|\,ds.
\]
Multiplying by \(e^{-\omega t}\) and applying Gronwall's inequality yields
\[
\|n(t)\|\le \|n_0\|\,e^{C_* t}\qquad\text{for all }t\in[0,T_{\max}),
\]
where \(C_*:=\omega+bR_{\max}C_\mathcal{N}\) is independent of \(t\) and \(T_{\max}\). Suppose \(T_{\max}<\infty\). Then \eqref{eq:range-of-R(t)} and the above estimate give
\[
\|(n,R)\|_{X_\tau}\le M_0:=R_{\max}+\|n_0\|\,e^{C_* T_{\max}}<\infty\qquad\text{for every }\tau\in[0,T_{\max}).
\]
By Theorem~\ref{thm:local-wellposedness}, the local existence time \(\tau\) is determined by a norm bound on the initial data and the model parameters. Choosing \(t_*<T_{\max}\) with \(T_{\max}-t_*<\tau(M_0)\) and applying Theorem~\ref{thm:local-wellposedness} with initial data \((n(t_*),R(t_*))\) extends the solution beyond \(T_{\max}\), contradicting maximality. Hence \(T_{\max}=\infty\).
\end{proof}

\section{Asymptotic Behavior}\label{sec:AsymptoticBehavior}

We now turn our attention to the steady state version of \eqref{eq:semipde1}, 
\begin{subequations}\label{eq:steadystate}
    \begin{align}
        0&=\chi_\alpha[b\hat{R}-d]\hat{n}-\mu b\chi_\alpha\hat{R}\hat{n}+b\mu\hat{R}\int^{1}_{0}\chi_\alpha(y)\hat{n}(y)p(x;y)\,dy-\partial_x[v\hat{n}]+m\partial_{xx}\hat{n}\label{eq:ss_n}\\
    0&=\theta-\eta \hat{R}-b\hat{R}\int^{1}_{0} \chi_\alpha(x) \hat{n}(x)\,dx.\label{eq:ss_R}
    \end{align}
\end{subequations}

The steady state equation \eqref{eq:ss_n} may be written as \(\mathcal L_{\hat R}\hat n=0\), where
\begin{equation}\label{eq:LR_operator}
\mathcal L_R := \mathcal A + bR\,\mathcal N,\qquad R\ge 0.
\end{equation}
Any nonnegative steady state of \eqref{eq:semipde1} thus corresponds to a nonnegative element of \(\ker(\mathcal L_{\hat R})\) for some \(\hat R\ge 0\). The long-time behavior of solutions is therefore governed by the spectral properties of the family \((\mathcal L_R)_{R\ge 0}\), and in particular by the sign of the spectral bound \(s(\mathcal L_R)\). We begin with two properties of \(\mathcal{A}\) not required for well-posedness but essential for the spectral analysis.

\begin{lem}[Analytic semigroup generated by \(\mathcal{A}\)]
\label{lem:Analytic-A}
The operator \(\mathcal A\) generates an analytic \(C_0\)-semigroup \(\{e^{\mathcal A t}\}_{t\ge 0}\) on \(L^2(\Omega)\).
\end{lem}

\begin{proof}
Write $\mathcal{A}=A_0+Q$, where $A_0:=m\partial_{xx}$ with domain $\mathcal{D}(A_0)=\{n\in H^2(\Omega):\partial_x n(0)=\partial_x n(1)=0\}$ and $Q:=-\partial_x(v(x)\,\cdot\,)-\chi_\alpha d$. The operator $A_0$ is self-adjoint on $L^2(\Omega)$ and satisfies $\langle A_0 n,n\rangle=-m\|\partial_x n\|_{L^2}^2\le 0$ for all $n\in\mathcal{D}(A_0)$, and hence its spectrum lies in $(-\infty,0]$. It is therefore sectorial and generates a bounded analytic semigroup on $L^2(\Omega)$ (see \cite[Corollary~II.4.7]{Engel-Nagel}).

We now verify that $Q$ is $A_0$-bounded with $A_0$-bound $0$. Since $\partial_x(v(x)n)=v(x)\partial_x n+v'(x)n$, every $n\in\mathcal{D}(A_0)$ satisfies
\begin{align}
\label{eq:Q-bound}
\|Qn\|_{L^2}\le\|v\|_\infty\|\partial_x n\|_{L^2}+(\|v'\|_\infty+d)\|n\|_{L^2}.
\end{align}
It therefore suffices to show that for every $\epsilon>0$ there exists $C_\epsilon>0$ such that $\|\partial_x n\|_{L^2}\le\epsilon\|A_0 n\|_{L^2}+C_\epsilon\|n\|_{L^2}$. Since $n\in\mathcal{D}(A_0)$ satisfies the Neumann conditions $\partial_x n(0)=\partial_x n(1)=0$, integration by parts yields
\[
\|\partial_x n\|_{L^2}^2=-\int_0^1 n\,\partial_{xx}n\,dx=\frac{1}{m}\langle n,-A_0 n\rangle\le\frac{1}{m}\|n\|_{L^2}\|A_0 n\|_{L^2}.
\]
For any $\delta>0$, Young's inequality gives
\[
\|\partial_x n\|_{L^2}^2\le\frac{\delta}{2m}\|A_0 n\|_{L^2}^2+\frac{1}{2m\delta}\|n\|_{L^2}^2.
\]
Taking square roots and choosing $\delta=2m\epsilon^2$, we obtain
\begin{align}
\label{eq:interpolation}
\|\partial_x n\|_{L^2}\le\epsilon\|A_0 n\|_{L^2}+\frac{1}{2m\epsilon}\|n\|_{L^2}.
\end{align}
Inserting \eqref{eq:interpolation} into \eqref{eq:Q-bound} gives
\[
\|Qn\|_{L^2}\le\|v\|_\infty\epsilon\,\|A_0 n\|_{L^2}+\left(\frac{\|v\|_\infty}{2m\epsilon}+\|v'\|_\infty+d\right)\|n\|_{L^2}.
\]
Since $\|v\|_\infty\epsilon$ can be made arbitrarily small by the choice of $\epsilon$, the operator $Q$ has $A_0$-bound $0$. Therefore, the operator $\mathcal{A}=A_0+Q$ generates an analytic $C_0$-semigroup on $L^2(\Omega)$ by \cite[Theorem~III.2.10]{Engel-Nagel}.
\end{proof}

Combining the positivity and irreducibility established in Lemma~\ref{lem:Semigroup-A-Positive-Irreducible} with the analyticity just proved yields pointwise strict positivity of the semigroup \(e^{\mathcal{A}t}\).

\begin{cor}[Strict positivity of \(e^{\mathcal{A}t}\)]
\label{cor:strict-positivity-A}
For every \(0\le f\in L^2(\Omega)\) with \(f\not\equiv 0\) and every \(t > 0\),
\[
e^{\mathcal{A}t}f(x) > 0\quad\text{for a.e.\ }x\in\Omega.
\]
\end{cor}

\begin{proof}
The conclusion follows from \cite[Theorem~10.1.2]{arendtheat}.
\end{proof}

Since \(\mathcal{N}\) is bounded on \(L^2(\Omega)\) by Lemma~\ref{lem:N-bounded}, the operator \(bR\mathcal{N}\) is a bounded perturbation of \(\mathcal{A}\) for each fixed \(R\ge 0\). The analyticity established in Lemma~\ref{lem:Analytic-A} therefore carries over to the full family \((\mathcal{L}_R)_{R\ge 0}\).

\begin{lem}[Analytic semigroup]
\label{lem:LR-analytic}
For each fixed \(R\ge0\), \(\mathcal L_R\) generates an analytic \(C_0\)-semigroup on \(L^2(\Omega)\), and \(\mathcal{D}(\mathcal L_R)=\mathcal{D}(\mathcal A)\).
\end{lem}

\begin{proof}
Since \(bR\mathcal N\) is bounded on \(L^2(\Omega)\), it is \(\mathcal A\)-bounded with \(\mathcal A\)-bound \(0\). Hence \(\mathcal L_R=\mathcal A+bR\mathcal N\) with \(\mathcal{D}(\mathcal L_R)=\mathcal{D}(\mathcal A)\) generates an analytic \(C_0\)-semigroup on \(L^2(\Omega)\) by \cite[Theorem~III.2.10]{Engel-Nagel}.
\end{proof}

In addition, the spectral analysis in the later discussion relies on duality arguments involving the adjoint operators \(\mathcal{A}^*\) and \(\mathcal{L}_R^*\). Specifically, the adjoint eigenfunctions are essential for determining the sign of the spectral bound \(s(\mathcal{L}_R)\). We now identify the domains and actions of these operators.

\begin{lem}[Adjoint of \(\mathcal A\) and \(\mathcal L_R\)]
\label{lem:adjoint-A-LR}
The adjoint of \(\mathcal A\) is given by
\[
\mathcal A^*\phi = m\,\phi_{xx} + v(x)\phi_x - d\,\chi_\alpha\phi, \quad \mathcal{D}(\mathcal A^*) = \{\phi\in H^2(\Omega):\phi_x(0)=\phi_x(1)=0\}.
\]
Moreover, for each fixed \(R\ge0\), \(\mathcal L_R^* = \mathcal A^* + bR\,\mathcal N^*\) with \(\mathcal{D}(\mathcal L_R^*) = \mathcal{D}(\mathcal A^*)\), where
\[
(\mathcal N^*\phi)(x) = (1-\mu)\chi_\alpha(x)\phi(x) + \mu\,\chi_\alpha(x)\int_0^1 p(y;x)\phi(y)\,dy.
\]
In particular, \(\mathcal{D}(\mathcal L_R^*) \subset H^2(\Omega) \hookrightarrow C^1(\overline{\Omega})\).
\end{lem}

\begin{proof}
We first identify the adjoint of \(\mathcal N\). For \(n,\phi\in L^2(\Omega)\),
\begin{align*}
\langle \mathcal N n,\phi\rangle
&=
(1-\mu)\int_0^1 \chi_\alpha(x)n(x)\phi(x)\,dx
+
\mu\int_0^1\left(\int_0^1 \chi_\alpha(y)n(y)p(x;y)\,dy\right)\phi(x)\,dx \\
&=
\int_0^1 n(x)\left[(1-\mu)\chi_\alpha(x)\phi(x) + \mu\,\chi_\alpha(x)\int_0^1 p(y;x)\phi(y)\,dy\right]dx,
\end{align*}
which gives the stated formula for \(\mathcal N^*\phi\).

We next identify the adjoint of \(\mathcal A\). Let \(\phi \in \mathcal{D}(\mathcal A^*)\), and let \(g \in L^2(\Omega)\) be the element satisfying \(\langle \mathcal A n, \phi\rangle = \langle n, g\rangle\) for all \(n \in \mathcal{D}(\mathcal A)\). Taking \(n \in C_c^\infty(0,1)\), integration by parts yields
\[
\int_0^1 \bigl(m n_{xx}-\partial_x(vn)-d\chi_\alpha n\bigr)\phi\,dx = \int_0^1 n\bigl(m\phi_{xx}+v\phi_x-d\chi_\alpha\phi\bigr)\,dx,
\]
and hence in the distributional sense \(m\phi_{xx} + v\phi_x - d\chi_\alpha\phi = g \in L^2(\Omega)\). Since \(m > 0\), \(v \in C^1(\overline{\Omega})\), and \(\chi_\alpha \in L^\infty(\Omega)\), one-dimensional elliptic regularity gives \(\phi \in H^2(\Omega)\) (see \cite[Theorem~8.30]{san2022}).

Now let \(n \in \mathcal{D}(\mathcal A)\). Integration by parts gives
\[
\langle \mathcal A n,\phi\rangle = \int_0^1 n\bigl(m\phi_{xx}+v\phi_x-d\chi_\alpha\phi\bigr)\,dx - m\bigl[n\,\phi_x\bigr]_0^1,
\]
where the transport boundary term vanishes because \(v(0) = v(1) = 0\). The identity \(\langle \mathcal A n,\phi\rangle = \langle n, g\rangle\) then forces the boundary term \(m[n\,\phi_x]_0^1\) to vanish for all \(n \in \mathcal{D}(\mathcal A)\). Since \(n(0)\) and \(n(1)\) are arbitrary for functions in \(\mathcal{D}(\mathcal A)\), we conclude that \(\phi_x(0) = \phi_x(1) = 0\), and therefore
\[
\mathcal{D}(\mathcal A^*) \subset \{\phi \in H^2(\Omega) : \phi_x(0) = \phi_x(1) = 0\}.
\]

Conversely, if \(\phi \in H^2(\Omega)\) with \(\phi_x(0) = \phi_x(1) = 0\), repeating the integration by parts shows that
\[
\langle \mathcal A n,\phi\rangle = \langle n, m\phi_{xx} + v\phi_x - d\chi_\alpha\phi\rangle
\]
for every \(n \in \mathcal{D}(\mathcal A)\), and therefore \(\phi \in \mathcal{D}(\mathcal A^*)\) with \(\mathcal A^*\phi = m\phi_{xx} + v\phi_x - d\chi_\alpha\phi\). This proves the formula and domain for \(\mathcal A^*\).

Finally, since \(bR\mathcal N\) is bounded on \(L^2(\Omega)\), we have \(\mathcal{D}(\mathcal L_R^*) = \mathcal{D}(\mathcal A^*)\) and \(\mathcal L_R^* = \mathcal A^* + bR\,\mathcal N^*\). For the forward inclusion, if \(\phi \in \mathcal{D}(\mathcal L_R^*)\), then for all \(n \in \mathcal{D}(\mathcal A)\),
\[
\langle \mathcal A n, \phi\rangle = \langle \mathcal L_R n, \phi\rangle - \langle bR\mathcal N n, \phi\rangle = \langle n, \mathcal L_R^*\phi - bR\mathcal N^*\phi\rangle,
\]
hence \(\phi \in \mathcal{D}(\mathcal A^*)\). Conversely, if \(\phi \in \mathcal{D}(\mathcal A^*)\), then \(\langle \mathcal L_R n, \phi\rangle = \langle \mathcal A n, \phi\rangle + \langle bR\mathcal N n, \phi\rangle = \langle n, \mathcal A^*\phi + bR\mathcal N^*\phi\rangle\) for all \(n \in \mathcal{D}(\mathcal A)\), which gives \(\phi \in \mathcal{D}(\mathcal L_R^*)\). The final assertion follows from the Sobolev embedding \(H^2(\Omega) \hookrightarrow C^1(\overline{\Omega})\) (see \cite[Theorem~5.3.1]{hillen2023}).
\end{proof}

We can now state the spectral properties of \((\mathcal L_R)_{R\ge 0}\).

\begin{prop}
\label{prop:spectral-sign-LR}
Assume \(F_1\)--\(F_3\). Then the following assertions hold for the family \((\mathcal L_R)_{R\ge 0}\).

\begin{enumerate}
\item[(i)] For every \(R\ge 0\), the operator \(\mathcal{L}_{R}\) has a compact and positive resolvent.

\item[(ii)] For every \(R\ge 0\), the spectral bound \(s(\mathcal{L}_{R})\) is a simple real eigenvalue of \(\mathcal{L}_{R}\) admitting a strictly positive eigenfunction. Likewise, \(s(\mathcal{L}_{R})\) is an eigenvalue of the adjoint operator \(\mathcal{L}_{R}^*\) with a strictly positive eigenfunction.

\item[(iii)] \(s(\mathcal L_{d/b}) = 0\).

\item[(iv)] The map \(R \mapsto s(\mathcal{L}_R)\) is strictly increasing on \([0,\infty)\).
\end{enumerate}

Consequently,
\[
s(\mathcal{L}_R)
\begin{cases}
< 0, & R < d/b, \\
= 0, & R = d/b, \\
> 0, & R > d/b.
\end{cases}
\]
\end{prop}

\begin{proof}
Since \(\mathcal{D}(\mathcal L_R)=\mathcal{D}(\mathcal A)=\{n\in H^2(\Omega):\partial_x n(0)=\partial_x n(1)=0\}\subset H^2(\Omega)\) and the embedding \(H^2(\Omega)\hookrightarrow L^2(\Omega)\) is compact, \(\mathcal L_R\) has compact resolvent by \cite[Proposition~II.4.25]{Engel-Nagel}.

By Lemma~\ref{lem:Semigroup-A-Positive-Irreducible}, \(\mathcal A\) generates a positive \(C_0\)-semigroup on \(L^2(\Omega)\), and by Lemma~\ref{lem:N-bounded}, \(bR\mathcal N\) is a bounded positive operator on \(L^2(\Omega)\). Hence \(\mathcal L_R\) generates a positive \(C_0\)-semigroup \((e^{[\mathcal L_R]t})_{t\ge 0}\) dominating \((e^{\mathcal A t})_{t\ge 0}\) by \cite[Corollary~VI.1.11]{Engel-Nagel}:
\begin{equation}\label{eq:semigroup-domination}
0\le e^{\mathcal A t}f\le e^{[\mathcal L_R]t}f\qquad\text{for all }f\in L^2_+(\Omega),\ t\ge 0.
\end{equation}
The positivity of the resolvent of \(\mathcal L_R\) then follows from the integral representation of the resolvent \cite[Theorem~II.1.10]{Engel-Nagel}
\[
(\lambda I-\mathcal{L}_R)^{-1}f=\int_0^\infty e^{-\lambda t}e^{[\mathcal L_R]t}f\,dt,\qquad \lambda>\omega_0(\mathcal{L}_R),
\]
since the integrand is nonnegative for \(f\ge 0\). This proves (i).

The irreducibility of \((e^{[\mathcal{L}_{R}]t})_{t\ge 0}\) follows from that of \((e^{\mathcal{A}t})_{t\ge 0}\) in Lemma~\ref{lem:Semigroup-A-Positive-Irreducible}, together with the domination \eqref{eq:semigroup-domination}: for every \(0\le f\in L^2(\Omega)\) with \(f\not\equiv 0\), every \(t>0\), and every measurable \(U\subset\Omega\) with \(|U|>0\),
\[
\int_{U}e^{[\mathcal{L}_{R}]t}f\,dx\ge\int_{U}e^{\mathcal{A}t}f\,dx>0,
\]
where the strict inequality follows from Corollary~\ref{cor:strict-positivity-A}. Hence \((e^{[\mathcal L_R]t})_{t\ge0}\) is positive and irreducible.

Combined with the compact resolvent from (i), the Krein--Rutman theorem implies that \(s(\mathcal L_R)\) is an eigenvalue of \(\mathcal L_R\) with a strictly positive eigenfunction \(\varphi_R>0\), and likewise an eigenvalue of \(\mathcal L_R^*\) with a strictly positive eigenfunction \(\psi_R>0\) (see \cite[Theorem~10.2.5]{arendtheat}). The simplicity of \(s(\mathcal L_R)\), with \(\ker(\mathcal L_R - s(\mathcal L_R)I) = \operatorname{span}\{\varphi_R\}\), follows from \cite[Theorem~10.2.6]{arendtheat}. This proves (ii).

We now prove (iii). We first identify the action of the adjoint \(\mathcal A^*\) on the constant function \(\mathbf 1 \equiv 1\). For every \(n\in\mathcal{D}(\mathcal A)\),
\[
\langle \mathcal A n,\mathbf 1\rangle = \int_0^1 \bigl(-d\,\chi_\alpha n - \partial_x(vn) + m\,n_{xx}\bigr)\,dx = -d\int_0^1 \chi_\alpha(x)n(x)\,dx = \langle n, -d\chi_\alpha\rangle,
\]
by the assumption \(F_1\) and boundary condition \eqref{eq:semiinicon}. Similarly, for the nonlocal operator \(\mathcal N\), we have for every \(n \in L^2(\Omega)\),
\begin{align*}
\langle \mathcal N n,\mathbf 1\rangle 
&=(1-\mu)\int_0^1 \chi_\alpha(x)n(x)\,dx+\mu\int_0^1 \chi_\alpha(y)n(y)\Bigl(\int_0^1 p(x;y)\,dx\Bigr)\,dy \\
&=\int_0^1 \chi_\alpha(y)n(y)\,dy=\langle n,\chi_\alpha\rangle,
\end{align*}
by assumption \(F_2\). It follows that \(\mathbf 1 \in \mathcal{D}(\mathcal A^*)\) with
\begin{align}
\label{eq:adjoint-on-1}
\mathcal A^* \mathbf 1 = -d\chi_\alpha \qquad\text{and}\qquad \mathcal N^* \mathbf 1 = \chi_\alpha.
\end{align}
Hence, for \(R = d/b\),
\begin{align}
\label{eq:L^*R=d/b1=0}
\mathcal L_{d/b}^* \mathbf 1 = \mathcal A^* \mathbf 1 + d\,\mathcal N^* \mathbf 1 = -d\chi_\alpha + d\chi_\alpha = 0.
\end{align}
Let \(\varphi_{d/b} > 0\) satisfy \(\mathcal L_{d/b}\varphi_{d/b} = s(\mathcal L_{d/b})\varphi_{d/b}\). Pairing with \(\mathbf 1\) and using \eqref{eq:L^*R=d/b1=0} gives
\[
s(\mathcal L_{d/b})\int_0^1 \varphi_{d/b}(x)\,dx = \langle\mathcal L_{d/b}\varphi_{d/b}, \mathbf 1\rangle = \langle\varphi_{d/b}, \mathcal L_{d/b}^* \mathbf 1\rangle = 0.
\]
Since \(\varphi_{d/b} > 0\), \(\int_0^1 \varphi_{d/b}(x)\,dx > 0\), and hence \(s(\mathcal L_{d/b}) = 0\). This proves (iii).

It remains to prove (iv). Let \(R_2>R_1\ge 0\), write \(s_i:=s(\mathcal L_{R_i})\), and let \(\varphi_1, \psi_2 > 0\) satisfy
\[
\mathcal L_{R_1}\varphi_1 = s_1\varphi_1 \quad\text{and}\quad \mathcal L_{R_2}^*\psi_2 = s_2\psi_2.
\]
Since \(\mathcal L_{R_2} - \mathcal L_{R_1} = b(R_2 - R_1)\mathcal N\), applying this difference to \(\varphi_1\) gives
\[
\mathcal L_{R_2}\varphi_1 - s_1\varphi_1 = b(R_2-R_1)\mathcal N\varphi_1.
\]
Taking the \(L^2\) inner product with \(\psi_2\) and using
\[
\langle \mathcal L_{R_2}\varphi_1, \psi_2\rangle = \langle\varphi_1, \mathcal L_{R_2}^*\psi_2\rangle = s_2\langle\varphi_1,\psi_2\rangle,
\]
we obtain
\begin{align}
\label{eq:s2-s1}
(s_2-s_1)\langle \varphi_1,\psi_2\rangle=b(R_2-R_1)\langle \mathcal N\varphi_1,\psi_2\rangle.
\end{align}
Since \(\varphi_1,\psi_2>0\), one has \(\langle \varphi_1,\psi_2\rangle>0\). Moreover, \(\mathcal N\varphi_1\ge 0\) by the positivity of \(\mathcal N\), and
\[
\int_0^1 (\mathcal N\varphi_1)(x)\,dx = \int_0^1 \chi_\alpha(x)\varphi_1(x)\,dx > 0
\]
shows \(\mathcal N\varphi_1 \not\equiv 0\); therefore \(\langle \mathcal N\varphi_1,\psi_2\rangle > 0\) and \eqref{eq:s2-s1} yields \(s_2 - s_1 > 0\). This proves that \(R\mapsto s(\mathcal{L}_R)\) is strictly increasing. The sign conclusions follow immediately from strict monotonicity together with \(s(\mathcal L_{d/b})=0\).
\end{proof}

\begin{cor}[Exponential stability]
\label{cor:exp-stability-LR-below-threshold}
Let \(R_*<d/b\). Then there exists a constant \(M\ge 1\) such that
\[
\|e^{[\mathcal{L}_{R_*}]t}\|_{\mathcal L(L^2(\Omega))}\le M e^{s(\mathcal{L}_{R_*}) t}\qquad\text{for all }t\ge 0.
\]
\end{cor}

\begin{proof}
By Proposition~\ref{prop:spectral-sign-LR}, \(R_* < d/b\) gives \(s(\mathcal L_{R_*}) < 0\). Since \(\mathcal{L}_{R_*}\) generates an analytic semigroup by Lemma \ref{lem:LR-analytic}, its growth bound equals its spectral bound by \cite[Corollary IV.3.12]{Engel-Nagel}. Therefore the semigroup generated by \(\mathcal{L}_{R_*}\) is exponentially stable, and the stated estimate follows (see \cite[Proposition VI.1.14 \& Theorem VI.1.15]{Engel-Nagel}).
\end{proof}

For the remainder of this section we denote by 
\[
N(t):=\int_0^1 n(x,t)\,dx,\quad N_\alpha(t):=\int_0^1 \chi_\alpha(x)n(x,t)\,dx
\]
the total and active (expression level below \(\alpha\)) biomass, respectively. Since $0\le\chi_\alpha\le 1$, one has $0\le N_\alpha(t)\le N(t)$ for all $t\ge 0$.

The following lemma shows that whenever $N_\alpha(t)\to 0$, the resource converges to its washout value. This fact is used in the proofs of both washout stability and weak persistence.

\begin{lem}[Resource convergence]
\label{lem:R-convergence}
Let \((n,R)\) be a nonnegative mild solution of \eqref{eq:semipde1}--\eqref{eq:semiinicon}. If
\(
N_\alpha(t) \longrightarrow 0\) as \(t\to\infty,
\)
then \(R(t)\to \theta/\eta\) as \(t\to\infty\).
\end{lem}

\begin{proof}
By \eqref{eq:range-of-R(t)}, \(R(t)\le R_{\max}\) for all \(t\ge 0\). Set \(h(t):=bR(t)N_\alpha(t)\ge 0\); then \(0\le h(t)\le bR_{\max}N_\alpha(t)\), and \(h(t)\to 0\) as \(t\to\infty\) by hypothesis. Variation of constants applied to the scalar equation \(\dot R=\theta-\eta R-h(t)\) yields
\begin{align}
\label{eq:VoC-R}
R(t)=R_0\,e^{-\eta t}+\frac{\theta}{\eta}\bigl(1-e^{-\eta t}\bigr)-\int_{0}^{t}e^{-\eta(t-s)}h(s)\,ds.
\end{align}
The first two terms tend to \(0\) and \(\theta/\eta\) respectively. For the convolution, fix \(\varepsilon>0\) and pick \(T_\varepsilon\) such that \(h(s)<\varepsilon\) for \(s\ge T_\varepsilon\). By \eqref{eq:range-of-R(t)}, \(h(s)\le bR_{\max}N_\alpha(s)\) for all \(s\ge 0\), and hence \(H:=\sup_{s\in[0,T_\varepsilon]}h(s)<\infty\). Splitting the integral at \(T_\varepsilon\),
\[
\int_0^t e^{-\eta(t-s)}h(s)\,ds\le H\,e^{-\eta(t-T_\varepsilon)}\frac{1-e^{-\eta T_\varepsilon}}{\eta}+\frac{\varepsilon}{\eta}\bigl(1-e^{-\eta(t-T_\varepsilon)}\bigr),
\]
and therefore \(\limsup_{t\to\infty}\int_0^t e^{-\eta(t-s)}h(s)\,ds\le \varepsilon/\eta\). Since this holds for every \(\varepsilon>0\), the \(\limsup\) is zero, and hence \(\int_0^t e^{-\eta(t-s)}h(s)\,ds\to 0\) as \(t\to\infty\). Hence \(R(t)\to \theta/\eta\).
\end{proof}

The asymptotic arguments below involve differentiating the map \(t\mapsto\langle n(t),\phi\rangle\) for \(\phi\in\mathcal{D}(\mathcal{A}^*)\). Since mild solutions need not lie in \(\mathcal{D}(\mathcal{A})\), the following lemma provides the required justification.

\begin{lem}[Differentiability along adjoint test functions]
\label{lem:weak-diff-mild}
Let \((n,R)\) be a nonnegative mild solution of \eqref{eq:semipde1}--\eqref{eq:semiinicon} on \([0,T)\). For every \(\phi\in\mathcal{D}(\mathcal{A}^*)\), the map \(t\mapsto \langle n(t),\phi\rangle\) is continuously differentiable on \((0,T)\) and satisfies
\[
\frac{d}{dt}\langle n(t),\phi\rangle = \langle n(t),\mathcal{A}^*\phi\rangle + bR(t)\langle \mathcal{N}n(t),\phi\rangle.
\]
\end{lem}

\begin{proof}
By the mild formula \eqref{eq:semi-mild-n} and the duality \(\langle e^{\mathcal{A}\sigma}g,\phi\rangle = \langle g, e^{\mathcal{A}^*\sigma}\phi\rangle\),
\[
\langle n(t),\phi\rangle = \langle n_0, e^{\mathcal{A}^*t}\phi\rangle + \int_0^t \langle bR(s)\,\mathcal{N}n(s), e^{\mathcal{A}^*(t-s)}\phi\rangle\,ds.
\]
Since \(\phi\in\mathcal{D}(\mathcal{A}^*)\), the map \(\sigma\mapsto e^{\mathcal{A}^*\sigma}\phi\) is continuously differentiable in \(L^2(\Omega)\) with derivative \(e^{\mathcal{A}^*\sigma}\mathcal{A}^*\phi\), and \(s\mapsto bR(s)\,\mathcal{N}n(s)\) is continuous in \(L^2(\Omega)\). Differentiating by the Leibniz rule gives
\[
\frac{d}{dt}\langle n(t),\phi\rangle = \langle n_0, e^{\mathcal{A}^*t}\mathcal{A}^*\phi\rangle + \langle bR(t)\,\mathcal{N}n(t),\phi\rangle + \int_0^t \langle bR(s)\,\mathcal{N}n(s), e^{\mathcal{A}^*(t-s)}\mathcal{A}^*\phi\rangle\,ds.
\]
Applying duality to the first and third terms and reassembling via the mild formula \eqref{eq:semi-mild-n} yields \(\langle n(t),\mathcal{A}^*\phi\rangle + bR(t)\langle\mathcal{N}n(t),\phi\rangle\), as claimed.
\end{proof}

\begin{thm}[Washout equilibrium]
\label{thm:washout}
Assume \(\theta/\eta < d/b\). Then \(\left(0,\frac{\theta}{\eta}\right)\) is the only nonnegative steady state of
\eqref{eq:semipde1}--\eqref{eq:semiinicon}. Moreover, for every nonnegative mild solution
\((n,R)\) of \eqref{eq:semipde1}--\eqref{eq:semiinicon},
\[
N(t) = \|n(\cdot,t)\|_{L^1(\Omega)}\longrightarrow 0
\qquad\text{and}\qquad
R(t)\longrightarrow \frac{\theta}{\eta}
\qquad\text{as }t\to\infty.
\]
The washout state is therefore globally asymptotically stable in the
\(L^1(\Omega)\times \mathbb R\) topology.
\end{thm}

\begin{proof}
We first show that the washout state is the only nonnegative steady state. Let \((\hat n,\hat R)\) be any nonnegative steady state solution and write \(\hat N_\alpha:=\int_0^1\chi_\alpha(x)\hat n(x)\,dx\). Integrating \eqref{eq:ss_n} over \(\Omega\) and using assumptions \(F_1\)--\(F_2\) and the boundary conditions \eqref{eq:semiinicon} gives \(0=(b\hat R-d)\hat N_\alpha\). Hence either \(\hat R = d/b\) or \(\hat N_\alpha = 0\). If \(\hat R = d/b\), the resource equation yields
\[
\hat N_\alpha = \frac{1}{d}\Bigl(\theta - \eta\,\frac{d}{b}\Bigr) = \frac{\eta}{d}\Bigl(\frac{\theta}{\eta} - \frac{d}{b}\Bigr) < 0,
\]
which is impossible since \(\hat n \ge 0\). Thus \(\hat N_\alpha = 0\), and hence \(\hat n \equiv 0\) on \((0,\alpha)\), and since \(\hat n\in \mathcal{D}(\mathcal A)\subset H^2(\Omega)\hookrightarrow C^1(\overline{\Omega})\), we have \(\hat n(\alpha)=0\). On \((\alpha,1)\) the steady state equation reduces to \(\partial_x(m\hat n' - v\hat n) = 0\), and hence
\[
m\hat n'(x) - v(x)\hat n(x) = C
\]
for some constant \(C\). Evaluating at \(x=1\) and using \(\hat n'(1)=0\) and \(v(1)=0\) gives \(C=0\), and therefore
\[
\hat n'(x) = \frac{v(x)}{m}\hat n(x) \qquad\text{on }(\alpha,1).
\]
Since \(v/m\) is continuous on \([\alpha,1]\), uniqueness for this first-order initial-value problem with \(\hat n(\alpha)=0\) gives \(\hat n \equiv 0\) on \([\alpha,1]\). Thus \(\hat n \equiv 0\) on \((0,1)\), and the resource equation gives \(\hat R = \theta/\eta\). Hence \((0, \theta/\eta)\) is the only nonnegative steady state.

We now prove global attraction. Let \((n,R)\) be any nonnegative mild solution of \eqref{eq:semipde1}--\eqref{eq:semiinicon}. By the resource equation,
\[
\dot R(t)=\theta-\eta R(t)-bR(t)N_\alpha(t) \le \theta-\eta R(t).
\]
Letting \(z\) solve \(\dot z = \theta - \eta z\) with \(z(0) = R(0)\), scalar comparison gives \(0\le R(t)\le z(t)\) for all \(t\ge 0\). Since \(z(t)\to \theta/\eta\) and \(\theta/\eta<d/b\), we may choose \(R_*\) with \(\theta/\eta<R_*<d/b\), so that there exists \(T_0>0\) with \(R(t)\le R_*\) for all \(t\ge T_0\).

By Proposition \ref{prop:spectral-sign-LR}, the map \(R\mapsto s(\mathcal{L}_R)\) is strictly increasing with \(s(\mathcal{L}_{d/b}) = 0\); since \(R_* < d/b\), it follows that \(s(\mathcal{L}_{R_*}) < 0\), and Corollary \ref{cor:exp-stability-LR-below-threshold} gives a constant \(M\ge 1\) such that
\[
\|e^{[\mathcal{L}_{R_*}]t}\|_{\mathcal L(L^2(\Omega))} \le Me^{s(\mathcal{L}_{R_*}) t}\qquad\text{for all }t\ge 0.
\]

Now fix \(t\ge T_0\). Since \(R(s)\le R_*\) for all \(s\ge T_0\), we may rewrite the equation \eqref{eq:semipde1} on \([T_0,\infty)\) as
\[
\partial_t n = \mathcal{L}_{R_*}n - b\bigl(R_*-R(t)\bigr)\mathcal N(n).
\]
Applying the variation-of-constants formula with generator \([\mathcal{L}_{R_*}]\) gives
\[
n(t)
=
e^{[\mathcal{L}_{R_*}](t-T_0)}n(T_0)
-
\int_{T_0}^{t}
e^{[\mathcal{L}_{R_*}](t-s)}
\,b\bigl(R_*-R(s)\bigr)\mathcal N(n(s))\,ds.
\]
Since \(n(s)\ge 0\), \(\mathcal N\) is positive, \(R_*-R(s)\ge 0\), and
\(e^{[\mathcal{L}_{R_*}]t}\) is a positive semigroup, the integral term is nonnegative.
Therefore,
\[
0\le n(t)\le e^{[\mathcal{L}_{R_*}](t-T_0)}n(T_0)
\qquad\text{for all }t\ge T_0.
\]
Taking \(L^2\) norms and using the semigroup estimate,
\[
\|n(\cdot,t)\|
\le
Me^{s(\mathcal{L}_{R_*})(t-T_0)}\|n(\cdot,T_0)\|
\longrightarrow 0
\qquad\text{as }t\to\infty.
\]
Since \(|\Omega|=1\), we have \(L^2(\Omega)\hookrightarrow L^1(\Omega)\) and hence
\[
N(t) = \|n(\cdot,t)\|_{L^1(\Omega)}
\le
\|n(\cdot,t)\|
\longrightarrow 0.
\]

It remains to show that \(R(t)\to \theta/\eta\). Since \(N_\alpha(t)\le N(t)\to 0\), Lemma~\ref{lem:R-convergence} yields \(R(t)\to \theta/\eta\) as \(t\to\infty\). Combining with \(N(t)\to 0\), we obtain
\begin{align}
\label{eq:N(t)-R(t)-to-0}
N(t)+\left|R(t)-\frac{\theta}{\eta}\right|\longrightarrow 0,
\end{align}
which is the claimed global attraction.

We now verify Lyapunov stability in the \(L^1(\Omega)\times\mathbb R\) topology. With \(R_*\) as above, set \(C:=1+bR_*/\eta\). By \eqref{eq:adjoint-on-1} and Lemma~\ref{lem:weak-diff-mild} with \(\phi=\mathbf 1\),
\begin{align}
\frac{d}{dt}N(t) &= \langle n(t),\mathcal{A}^*\mathbf{1}\rangle + bR(t)\langle \mathcal{N}n(t),\mathbf{1}\rangle \nonumber\\
&= \langle n(t),-d\,\chi_\alpha\rangle + bR(t)\langle n(t),\mathcal{N}^*\mathbf{1}\rangle \nonumber\\
&= -d\,N_\alpha(t) + bR(t)\,N_\alpha(t) \nonumber\\
&= (bR(t)-d)\,N_\alpha(t). \label{eq:dNdt}
\end{align}
Given \(\varepsilon>0\), choose \(\delta:=\min\bigl\{\varepsilon/(C+1),\,R_*-\theta/\eta\bigr\}\) and suppose \(\|n_0\|_{L^1(\Omega)}<\delta\) and \(|R_0-\theta/\eta|<\delta\). The second constraint gives \(R_0<\theta/\eta+\delta\le R_*\), and \eqref{eq:range-of-R(t)} yields \(R(t)\le\max\{R_0,\theta/\eta\}\le R_*<d/b\) for all \(t\ge 0\). Hence \(bR(t)-d<0\), and since \(N_\alpha(t)\ge 0\), \eqref{eq:dNdt} gives \(\frac{d}{dt}N(t)\le 0\), and therefore
\[
N(t) = \|n(\cdot,t)\|_{L^1(\Omega)}\le \|n_0\|_{L^1(\Omega)}\quad\text{for all }t\ge 0.
\]
Since \(0\le bR(t)N_\alpha(t)\le bR_*\|n_0\|_{L^1(\Omega)}\), subtracting \(\theta/\eta\) from both sides of \eqref{eq:VoC-R} and taking absolute values gives
\begin{align*}
\Bigl|R(t)-\frac{\theta}{\eta}\Bigr| &= \Bigl|\Bigl(R_0-\frac{\theta}{\eta}\Bigr)e^{-\eta t} - \int_0^t e^{-\eta(t-s)}bR(s)N_\alpha(s)\,ds\Bigr| \\
&\le \Bigl|R_0-\frac{\theta}{\eta}\Bigr|e^{-\eta t} + \int_0^t e^{-\eta(t-s)}bR(s)N_\alpha(s)\,ds \\
&\le \Bigl|R_0-\frac{\theta}{\eta}\Bigr| + bR_*\|n_0\|_{L^1(\Omega)}\int_0^t e^{-\eta(t-s)}\,ds \\
&\le \Bigl|R_0-\frac{\theta}{\eta}\Bigr| + \frac{bR_*}{\eta}\|n_0\|_{L^1(\Omega)}.
\end{align*}
Combining these two estimates,
\[
N(t)+\Bigl|R(t)-\frac{\theta}{\eta}\Bigr|\le (C+1)\delta\le\varepsilon\quad\text{for all }t\ge 0.
\]
Together with the global attraction established above, this proves global asymptotic stability of the washout equilibrium in \(L^1(\Omega)\times\mathbb R\).
\end{proof}

Theorem \ref{thm:washout} has shown that the system \eqref{eq:steadystate} converges to the washout equilibrium when \(\theta/\eta < d/b\). We now analyze the asymptotic behavior when the opposite inequality, \(\theta/\eta > d/b\), holds.

\begin{thm}[Positive equilibrium above threshold]
\label{thm:positive-equilibrium}
Assume \(\theta/\eta > d/b\). Then the steady state system \eqref{eq:steadystate} admits exactly two nonnegative equilibria:
\[
(\hat n,\hat R)=\left(0,\frac{\theta}{\eta}\right)
\qquad\text{and}\qquad
(\hat n,\hat R)=\left(\hat n^*,\frac{d}{b}\right),
\]
where \(\hat n^*>0\).
\end{thm}

\begin{proof}
The washout equilibrium \((0, \theta/\eta)\) is an immediate solution of \eqref{eq:steadystate}. Now let \((\hat n,\hat R)\) be any nonnegative steady state. Then \eqref{eq:ss_n} may be written as \(\mathcal L_{\hat R}\hat n = 0\). Taking the \(L^2\) inner product of \eqref{eq:ss_n} with the constant function \(\mathbf 1\), and using assumptions \(F_1\)--\(F_2\) and the boundary conditions \eqref{eq:semiinicon}, we obtain
\[
0 = (b\hat R - d)\hat N_\alpha.
\]
Hence either \(\hat N_\alpha = 0\) or \(\hat R = d/b\). If \(\hat n \equiv 0\), then \eqref{eq:ss_R} gives \(\hat R = \theta/\eta\), recovering the washout equilibrium.

Assume now that \(\hat n \not\equiv 0\). Since \(\mathcal L_{\hat R}\hat n = 0\), the strictly positive adjoint eigenfunction \(\psi_{\hat R}>0\) from Proposition~\ref{prop:spectral-sign-LR}(ii) yields
\[
0 = \langle\mathcal L_{\hat R}\hat n,\psi_{\hat R}\rangle = \langle\hat n,\mathcal L_{\hat R}^*\psi_{\hat R}\rangle = s(\mathcal L_{\hat R})\langle\hat n,\psi_{\hat R}\rangle.
\]
Since \(\hat n\ge0\), \(\hat n\not\equiv 0\), and \(\psi_{\hat R}>0\) a.e., we have \(\langle\hat n,\psi_{\hat R}\rangle>0\), and hence \(s(\mathcal L_{\hat R})=0\). By the strict monotonicity and sign characterization in Proposition~\ref{prop:spectral-sign-LR}, this forces \(\hat R = d/b\). Substituting into \eqref{eq:ss_R} gives
\[
\hat N_\alpha = \frac{\theta - \eta(d/b)}{d} > 0,
\]
and hence \(\hat n\) is nontrivial on the active region. At \(R = d/b\), Proposition~\ref{prop:spectral-sign-LR}(ii) shows that \(\ker(\mathcal L_{d/b}) = \operatorname{span}\{\varphi_{d/b}\}\) with \(\varphi_{d/b} > 0\), and therefore every nonnegative solution of \(\mathcal L_{d/b}\hat n = 0\) has the form \(\hat n = c\,\varphi_{d/b}\) for some \(c > 0\). The constant \(c\) is uniquely determined by \eqref{eq:ss_R}, namely
\[
c\int_0^1 \chi_\alpha(x)\varphi_{d/b}(x)\,dx = \frac{\theta-\eta(d/b)}{d}.
\]
Thus, under the condition \(\theta/\eta > d/b\), there exists a unique positive equilibrium \((\hat n^*, d/b)\) in addition to the washout state.
\end{proof}

When \(\theta/\eta > d/b\), Proposition~\ref{prop:spectral-sign-LR} gives \(s(\mathcal{L}_{\theta/\eta})>0\), suggesting that solutions should not converge to the washout state. The following theorem confirms this.

\begin{thm}[Weak persistence]
\label{thm:weak-persistence}
Assume \(\theta/\eta > d/b\), and let \((n,R)\) be a nonnegative mild solution of \eqref{eq:semipde1}--\eqref{eq:semiinicon} with \(n_0\not\equiv0\). Then
\[
\limsup_{t\to\infty} N(t)>0.
\]
That is, \(n\) is weakly persistent.
\end{thm}

\begin{proof}
We argue by contradiction. Choose \(\delta>0\) small enough that \(\theta/(\eta+b\delta)>d/b\), and suppose that \(\limsup_{t\to\infty}N(t)<\delta\). Pick \(R_-\) with \(d/b<R_-<\theta/(\eta+b\delta)\). Then there exists \(T_0>0\) such that \(N(t)\le\delta\) for all \(t\ge T_0\). Since \(N_\alpha(t)\le N(t)\le\delta\) for \(t\ge T_0\), the resource equation gives
\[
\dot R(t)=\theta-\eta R(t)-bR(t)N_\alpha(t)\ge\theta-(\eta+b\delta)R(t)\qquad\text{for all }t\ge T_0.
\]
Comparing with the scalar ODE \(\dot z=\theta-(\eta+b\delta)z\), whose solutions converge to \(\theta/(\eta+b\delta)\), we conclude that there exists \(t_1\ge T_0\) such that \(R(t)\ge R_-\) for all \(t\ge t_1\). By Proposition~\ref{prop:spectral-sign-LR}, \(s(\mathcal L_{R_-})>0\), and there exists \(\psi_-\in\mathcal{D}(\mathcal L_{R_-}^*)\) with \(\psi_->0\) satisfying
\[
\mathcal L_{R_-}^*\psi_-=s(\mathcal L_{R_-})\psi_-.
\]

Define \(F(t):=\langle n(\cdot,t),\psi_-\rangle\). We first justify that \(F(t_1)>0\). Since \(n_0\not\equiv 0\), the mild formula \eqref{eq:semi-mild-n} and the nonnegativity of the integrand give \(n(\cdot,t)\ge e^{\mathcal A t}n_0\) for all \(t>0\). Hence \(n(\cdot,t_1)>0\) a.e.\ on \(\Omega\) by Corollary~\ref{cor:strict-positivity-A}, and since \(\psi_->0\) a.e., it follows that \(F(t_1)>0\).

Since \(\psi_-\in\mathcal{D}(\mathcal{A}^*)\) by Lemma~\ref{lem:adjoint-A-LR}, Lemma~\ref{lem:weak-diff-mild} with \(\phi=\psi_-\) gives, for \(t\ge t_1\),
\begin{align*}
F'(t)
&= \langle n(\cdot,t),\mathcal{A}^*\psi_-\rangle + bR(t)\langle \mathcal N n(\cdot,t),\psi_-\rangle \\
&= \langle n(\cdot,t),\mathcal L_{R_-}^*\psi_- - bR_-\mathcal N^*\psi_-\rangle
+bR(t)\langle n(\cdot,t),\mathcal N^*\psi_-\rangle \\
&= \langle n(\cdot,t),\mathcal L_{R_-}^*\psi_-\rangle
+b(R(t)-R_-)\langle \mathcal N n(\cdot,t),\psi_-\rangle.
\end{align*}
Since \(R(t)\ge R_-\) for \(t\ge t_1\) and both \(\mathcal N\) and \(\psi_-\) are positive, the second term is nonnegative, and hence
\[
\frac{d}{dt}F(t) \ge s(\mathcal L_{R_-}) F(t) \qquad\text{for all }t \ge t_1.
\]
Gronwall's inequality yields
\begin{align}
\label{eq:lower-bound-F(t)}
F(t)\ge e^{s(\mathcal L_{R_-})(t-t_1)}F(t_1)\longrightarrow\infty\qquad\text{as }t\to\infty.
\end{align}

On the other hand, Lemma~\ref{lem:adjoint-A-LR} gives \(\psi_- \in \mathcal{D}(\mathcal L_{R_-}^*) \subset C^1(\overline{\Omega})\), and hence \(\psi_- \in L^\infty(\Omega)\). It follows that
\[
0 \le F(t) \le \|\psi_-\|_{L^\infty} N(t)\le\|\psi_-\|_{L^\infty}\,\delta\qquad\text{for all }t\ge T_0,
\]
which contradicts \eqref{eq:lower-bound-F(t)}. Hence \(\limsup_{t\to\infty} N(t) \ge \delta > 0\), and \(n\) is weakly persistent.
\end{proof}

\begin{rem}[Threshold independence]
\label{rem:threshold-independence}
The persistence threshold \(\theta/\eta = d/b\) depends only on the resource supply rate \(\theta\), the dilution rate \(\eta\), the death rate \(d\), and the growth rate coefficient \(b\). It is independent of all parameters governing the internal phenotypic dynamics: the phenotypic drift \(v\), the redistribution kernel \(p\), the switching probability \(\mu\), the persister cutoff \(\alpha\), and the diffusion coefficient \(m\). Biologically, this means that whether a bacterial population persists or goes extinct in the chemostat is determined entirely by the balance between nutrient supply and cell loss, not by the mechanism of phenotypic switching. The switching dynamics shape the phenotypic composition and the profile of the positive equilibrium \(\hat n^*\), but they cannot rescue a population from extinction when the resource supply is insufficient.
\end{rem}

This result shows that the population is weakly persistent, but the form of that persistence remains unknown. Based on the numerical results in the next section, we have the following conjecture. 

\begin{conjecture}
\label{conj:LinearLocalStability}
If \(\theta/\eta > d/b\), the unique positive steady state \((\hat n^*, d/b)\) from Theorem~\ref{thm:positive-equilibrium} is locally asymptotically stable.
\end{conjecture}

\section{Numerical Simulations} \label{sec:Numerics}

In this section, we simulate the model by approximating the PDE in equation \eqref{eq:semipde1} using a system of \(K\) ordinary differential equations. Code is available at \href{https://github.com/Tyler-Meadows/Persisters}{Tyler Meadows's GitHub repository}.

\begin{figure}[!t]
    \centering
    \begin{subfigure}[b]{0.48\linewidth}
        \centering
        \includegraphics[width=\linewidth]{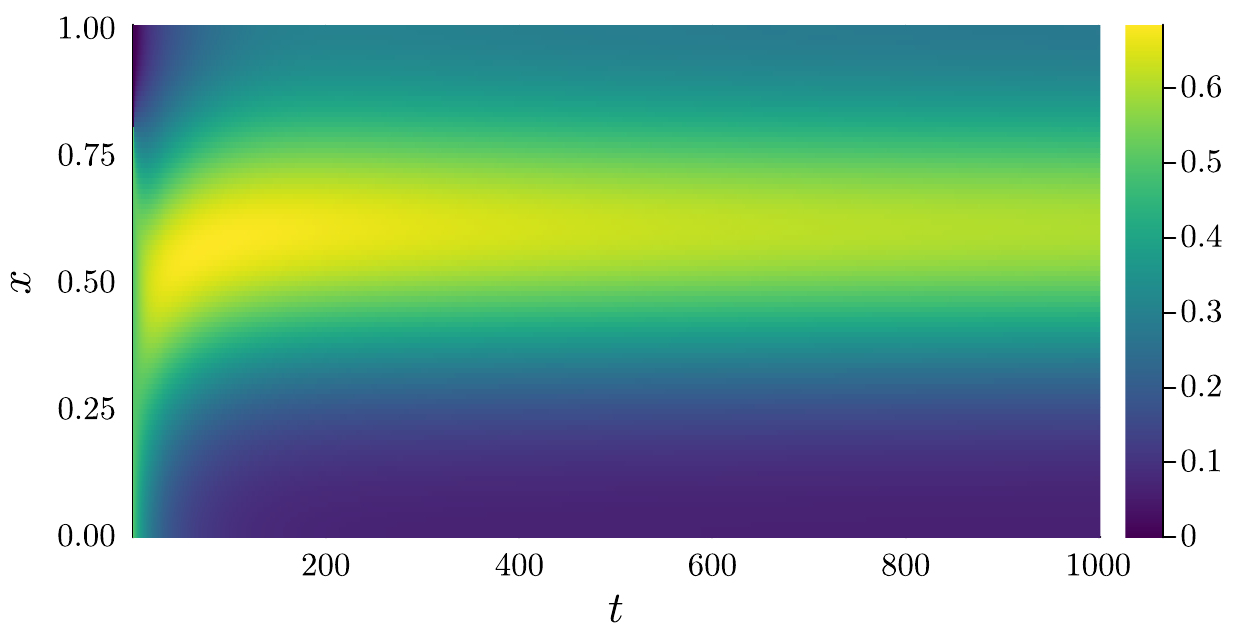}
        \label{fig:nopersisters}
    \end{subfigure}
    \hfill
    \begin{subfigure}[b]{0.48\linewidth}
        \centering
        \includegraphics[width=\linewidth]{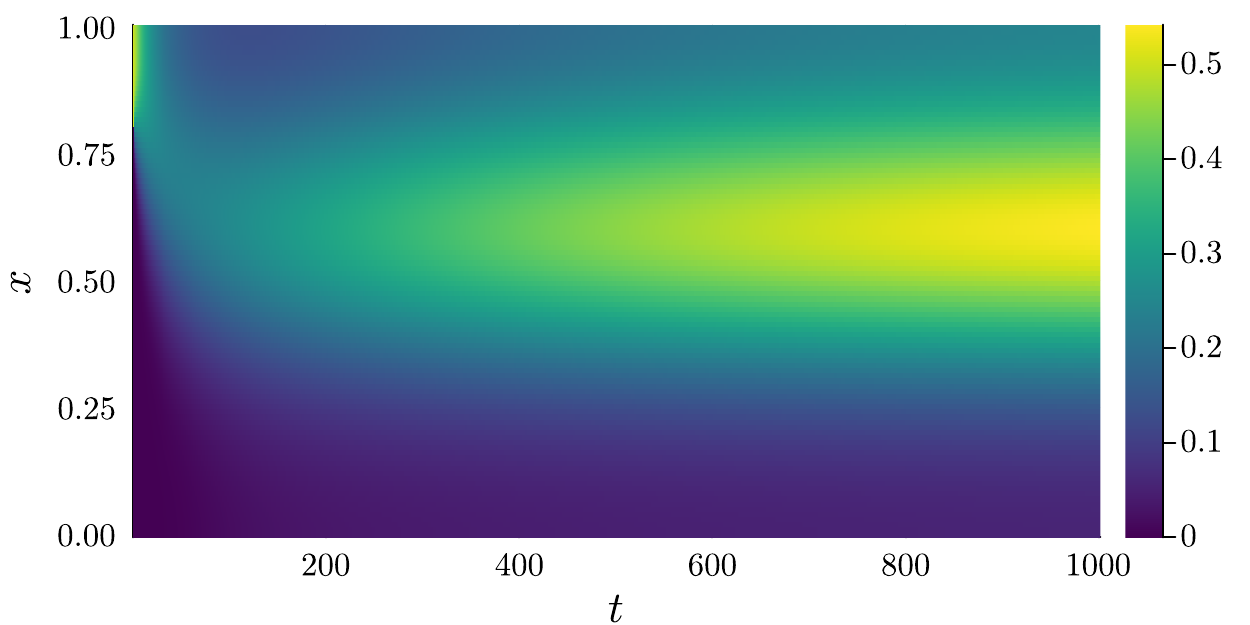}
        \label{fig:onlypersisters}
    \end{subfigure}
    \caption{Heatmaps showing the distribution of the persister phenotype $n(x,t)$ as a function of time under two initial conditions: (a) no persister cells ($n(x,0) = 0$ for $x > \alpha$); (b) only persister cells ($n(x,0) = 0$ for $x < \alpha$). Parameters used were $\theta = 1.0$, $\eta = 0.3$, $d = 0.03$, $b = 0.6$, $\mu = 0.4$, $\alpha = 0.80$, $\gamma = 0.6$, $m = 10^{-2}$, $v_0 = 1.0$.}
    \label{fig:persister_comparison}
\end{figure}

For simplicity, we assume that \(\alpha = k/K\) for some \(k, K \in \mathbb{N}\) with \(k < K\), so that no compartment contains both persister and regular cells. Each function may be approximated by its value at the midpoint of the compartment, denoted \(x_i\) for the \(i\)th compartment. Spatial derivatives are approximated by central finite differences,
\begin{align*}
\partial_x n(x_i) \approx \Delta n(x_i) &= \frac{n(x_{i+1})-n(x_{i-1})}{2/K},\\
\partial_{xx}n(x_i)\approx\Delta^{2}n(x_i) &= \frac{n(x_{i+1})-2n(x_i)+n(x_{i-1})}{(1/K)^{2}},
\end{align*}
with the Neumann boundary conditions imposed through ghost cells \(n_{0}:=n_{1}\) and \(n_{K+1}:=n_{K}\). The integral term is approximated by a sum,
\[
\int_0^1 p(x_i; y)\chi_\alpha(y)n(y)\,dy \approx \sum_{j=1}^{k} p(x_i; y_j) n_j \cdot \frac{1}{K}.
\]
Writing \(n(x_i, t) = n_i\), the finite-difference approximation of \eqref{eq:semipde1} reads
\begin{subequations}
    \begin{align}
        \partial_t n_i &= m \Delta^2 n_i - \Delta(v(x_i)n_i) + bR(1-\mu)n_i - dn_i + \mu b R \sum_{j=1}^k p(x_i; y_j)n_j \cdot \frac{1}{K} && i \le k, \\
        \partial_t n_i &= m \Delta^2 n_i - \Delta(v(x_i)n_i) + \mu b R \sum_{j=1}^k p(x_i; y_j)n_j \cdot \frac{1}{K} && i > k, \\
        \partial_t R &= \theta - \eta R - bR\sum_{i=1}^k n_i \cdot \frac{1}{K}.
    \end{align}
\end{subequations}

In our simulations we take \(p(x;y) \equiv 1\), so that daughter cells are uniformly distributed over all phenotypes. Figure~\ref{fig:persister_comparison} shows convergence to the positive steady state under two initial conditions: (a) no persister cells; (b) only persister cells. Time series simulations were done using the differential equations package in Julia \cite{rackauckas2017differentialequations}. These numerical results are consistent with Conjecture~\ref{conj:LinearLocalStability}; convergence from two markedly different initial conditions to the same steady state further suggests that global asymptotic stability may hold, a stronger conjecture we leave for future work.

\section*{Acknowledgments}

The authors thank Adrian~Lam and Thomas~Hillen for many fruitful discussions. T.~Day is supported by an NSERC Discovery Grant.  

\bibliographystyle{siamplain}
\bibliography{references}
\end{document}